\documentclass[10pt]{article}

\usepackage{amssymb,bbm,pifont,upgreek,graphicx,pgf,tikz
}

\usepackage[colorlinks=true]{hyperref}
\hypersetup{urlcolor=blue, citecolor=red}

\newcommand{\giu}{{\medskip\noindent}}
\newcommand{\Giu}{{\bigskip\noindent}}
\newcommand{\nl}{{\smallskip\noindent}}
\newcommand{\x}{ {\xi}   }
\newcommand{\p}{ {\pi}   }
\newcommand{\h}{ {\eta}   }
\newcommand{\vf}{ {\varphi}   }

\newcommand{\ee}{{\rm \bf e}}

\newcommand{\etar}{{\hat\eta}\,}
\newcommand{\nur}{{\hat\nu}\,}
\newcommand{\epsr}{{\hat\e}\,}
\usepackage[centertags]{amsmath}
\usepackage{amsthm}%questo serve per non numerare i teoremi
\usepackage{graphicx}
\font\teneufm=eufm10 \font\seveneufm=eufm7 \font\fiveeufm=eufm5
\newfam\eufmfam
\textfont\eufmfam=\teneufm \scriptfont\eufmfam=\seveneufm
\scriptscriptfont\eufmfam=\fiveeufm

\newtheorem{thm}{Theorem}[section]

\newtheorem{pro}[thm]{Proposition}
\newtheorem{lem}[thm]{Lemma}
\newtheorem{rem}[thm]{Remark}
\newcommand\brem{\begin{rem}\rm \small}
\newcommand\erem{\end{rem}}

\renewcommand{\proof}{\noi {\bf Proof.}\ }
\newcommand{\eproof}{\hskip.5truecm
\vrule width 1.7truemm height 3.5truemm depth 0.truemm
\par\Giu}

\newcommand{\e}{\varepsilon}
\renewcommand{\a}{\alpha}

\newcommand{\noi}{\noindent}

\newcommand\equ[1]{{\rm (\ref{#1})}}

\allowdisplaybreaks

\renewcommand{\Im}{{\, \rm Im \,}}
\newcommand\beq[1]{ \begin{equation}\label{#1} }
\newcommand{\eeq}{ \end{equation} }

\newcommand\beqa[1]{ \begin{eqnarray} \label{#1}}
\newcommand{\eeqa}{ \end{eqnarray} }
\newcommand{\beqano}{ \begin{eqnarray*} }
\newcommand{\eeqano}{ \end{eqnarray*} }

\newcommand\bB{{\mathbb{B}}}

\newcommand\ck[1]{{C^{#1}_{\rm per,0}}}
\newcommand\real{{\mathbb{R}}}

\newcommand\integer{{\mathbb{Z}}}

\newcommand\dst{\displaystyle}

\newcommand{\R}{\mathbbm R}
\newcommand{\Z}{\mathbbm Z}

\newtheorem{theorem}{Theorem}
\newtheorem{proposition}[theorem]{Proposition}
\newtheorem{lemma}[theorem]{Lemma}

\theoremstyle{definition}

\begin{document}

\title{\bf The spin--orbit resonances of the Solar system:\\
A mathematical treatment matching physical data\thanks{{\bf Acknowledgments.}
We thank J. Castillo-Rogez, A. Celletti, M. Efroimsky
and
 F. Nimmo
for useful discussions. \newline
Partially supported by  the MIUR
grant
``Critical Point Theory and Perturbative Methods for Nonlinear Differential Equations'' (PRIN2009).
}}
\author{
Francesco Antognini \\ \small
\vspace{-.2truecm}
{\scriptsize Departement Mathematik}
\\{\scriptsize  ETH-Z\"urich}
\vspace{-.2truecm}
\\{\scriptsize CH-8092 Z\"urich, Switzerland }
\vspace{-.2truecm}
\\{\scriptsize antognif@math.ethz.ch}
\and
Luca Biasco and Luigi Chierchia \\ \small
\vspace{-.2truecm}
{\scriptsize Dipartimento di Matematica e Fisica}
\\{\scriptsize  Universit\`a  ``Roma Tre''}
\vspace{-.2truecm}
\\{\scriptsize Largo S.L. Murialdo 1, I-00146 Roma (Italy) }
\vspace{-.2truecm}
\\{\scriptsize biasco,luigi@mat.uniroma3.it}
}

\date{\small Decemebr  9,  2013}

\maketitle

\vspace{-1.truecm}

\begin{abstract}{\footnotesize\noindent
In the mathematical framework of a  restricted, slightly dissipative spin--orbit model, we prove the existence of periodic orbits for astronomical parameter values corresponding to {\sl all} satellites of the Solar system observed in  exact spin--orbit resonance.
}
\end{abstract}

{\footnotesize
\nl
{{\bf Keywords:} Periodic orbits. Celestial Mechanics.  Spin--orbit resonances.
Moons in the Solar systems. Mercury. Dissipative systems.

\nl
}

}

\tableofcontents

\newpage

\section{Introduction and results}

\nl
{$\bullet$} {\it Satellites in spin--orbit resonance}

\nl
One of the many fascinating features of the Solar system is the presence of moons moving
in a ``synchronous'' way around their planets, as experienced, for example, by
earthlings looking always the same familiar face of their satellite. Indeed, eighteen moons
of our
Solar system
move in a so--called 1:1 spin--orbit resonance: while performing
a complete revolution on a (approximately) Keplerian ellipse around their principal
body, they also  complete a rotation around their spin axis (which is -- again, approximately
-- perpendicular to revolution plane), in this way these moons always show the same side
to their hosting planets.

\nl
The  list of the eighteen moons is the following:
  Moon (Earth);
  %Phobos, Deimos (Mars);
 Io, Europa, Ganymede, Callisto (Jupiter);
 %Amalthea (Jupiter);
 Mimas, Enceladus, Tethys, Dione, Rhea, Titan, Iapetus,
 %Janus, Epimetheus
(Saturn); Ariel, Umbriel, Titania, Oberon, Miranda (Uranus); Charon (Pluto);
minor bodies with mean radius smaller  than 100 Km are not considered (see, however, Appendix~\ref{minor bodies}).

\nl
There is only one more occurrence of spin--orbit resonance in the Solar system: the strange
case of the 3:2 resonance of Mercury around the Sun (i.e., Mercury
rotates three times on its spin axis, while making two orbital revolutions
around the Sun).

\nl In this paper we discuss a mathematical theory, which
%``proves'' 
%%% correction
is consistent with
the existence of all spin--orbit resonances of the
Solar system; in other words, we {\sl prove a theorem}, in a
framework of a well--known simple ``restricted spin--orbit
model'', establishing the existence of periodic orbits {\sl for
parameter values corresponding to   all the satellites 
%(moons and
%%% correction
(or
Mercury) in our Solar system} observed  in  spin--orbit resonance.

\nl
We remark that, in dealing with  mathematical models
trying to describe physical phenomena,  
one may be able to rigorously prove theorems only for
parameter values, typically, quite smaller than the physical ones;
on the other hand, for the true physical values, typically, one only obtains
numerical evidence.
In the present case, thanks
to sharp estimates, we are able to fill such a  gap and prove rigorous results for
the real parameter values.
Moreover, such results might also be an indication that 
the mathematical model adopted
is quite effective in describing the  physics.

\Giu
{$\bullet$} {\it The mathematical model}

\nl
We consider a  simple -- albeit  non trivial --  model in
which the center of mass of the satellite moves on a given
two--body Keplerian orbit focussed on a massive point (primary
body) exerting gravitational attraction on the body of the
satellite modeled by a triaxial ellipsoid with
equatorial  axes $a\ge b>0$ and polar axis $c$;  the spin
polar axis is assumed to be perpendicular to the Keplerian orbit
plane\footnote{The largest relative inclination (of the spin axis
on the orbital plane) is that of Iapetus ($8.298^{\circ}$)
followed by Mercury ($7^{\circ}$), Moon ($5.145^{\circ}$), Miranda
($4.338^{\circ}$); all the other moons have an inclination of
the order of one degree or less.}; finally, we include also small
dissipative effects (due to the possible internal
non--rigid structures of the satellite),  according to the  ``viscous--tidal model,
with a linear dependence on the tidal frequency'' (\cite{L}): essentially, the dissipative term is given by the average over one revolution period  of the so--called MacDonald's torque \cite{MD}; compare \cite{peale}.

\nl
For a discussion of  this model, see
\cite{Celletti90};
%(with   the normalization $n:=\sqrt{(Gm)/a^3}=1$);
for further references, see
\cite{Danby62}, \cite{Goldreich-Peale67}, \cite{Wisdom87_r}, and \cite{Celletti10_libro};
for a different (PDE) model, see \cite{BH1}.

%\nl Under the above hypotheses, the motions of the satellite are described by the
%angle $x$ formed by  the direction of (say) the major physical axis  of the
%satellite  with a fixed axis of the
%Keplerian orbit plane (say the direction of the semimajor axis of the ellipse).
%

 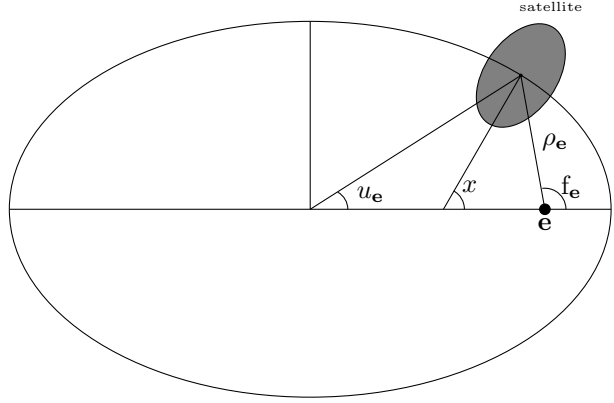
\begin{figure}
\centering
\begin{tikzpicture}

\draw (0,0) ellipse (4cm and 2.5cm);
\filldraw[fill opacity=0.5, fill=black, rotate around={-125:(2.8,1.78)}] (2.8,1.78) ellipse (0.77 cm and 0.48 cm);

%\filldraw[fill opacity=0.5,fill=yellow] (0,0) ellipse (4cm and 2.5cm);

\draw[-] (-4,0) -- (4,0) node[right] {$$} coordinate (x axis);
\draw[-] (0,0) -- (0,2.5) node[midway, left] {$$} coordinate (x axis);
\draw[-] (2.8,1.78) -- (3.12,0) node[midway, right] {$ \rho_{\textbf{e}}$};
%\draw[-] (2.8,1.78)  -- (2.415,1.114);
\draw[-] (2.8,1.78)  -- (1.771,0);
%\draw[->] (0,-2.5) -- (0,2.5) node[above] {$y$} coordinate (y axis);

\draw[-] (2.8,1.78)  -- (0,0);

\filldraw (3.12,0) circle (2pt) node[below] {$ \textbf{e} $}
(2.8,1.78) circle (0.5pt) node[below] {$$};

% \filldraw[fill=black!20!white, draw=black!50!black]
 %  (1.31,-.3) -- (1.34mm,0mm) arc (0:120:3mm) -- cycle;

\draw  (3.4,0) arc (0:97:0.28cm);
\draw  (2.05,0) arc (0:61:0.28cm);

\draw  (0.5,0) arc (0:61:0.26cm);
\draw (0.8,0.2) node {$u_\ee$};

%\draw (1.31,-.3) node {$\mathcal{P}$};
%\draw (2.8,1.78) node {$\mathcal{S}$};
\draw (3.47,0.3) node {$\textrm{f}_{\textbf{e}}$};
\draw (2.12,0.3) node {$x$};
\draw (3.2,2.7) node {$\tiny{\textrm{satellite}}$};

%\foreach \x in {-1,0,1,2}
%\draw (\x cm,1pt) -- (\x cm,-1pt) node[anchor=north] {$\x$};
%\foreach \y in {-2,-1,1,2}
%\draw (1pt,\y cm) -- (-1pt,\y cm) node[anchor=east] {$\y$};

%\draw (-1,1) .. controls (-1,1) and (2,-2) .. node[near end, left] {$y=-x$} (2,-2);
%\draw (-1,1) .. controls (0,3) and (1,2) .. node[right] {$y=2-x^2$} (2,-2);

\end{tikzpicture}

\caption{\footnotesize Triaxial satellite revolving on a rescaled Keplerian ellipse (equatorial section)}
\end{figure}

\Giu
The differential equation governing the motion of the satellite  is then given by
\begin{equation}
\label{eqn1}
\ddot{x} + {\eta} (\dot{x} - {\nu})+{\varepsilon}f_x(x,t)=0 \ ,
\end{equation}
where:
\begin{enumerate}
 \item[(a)] $x$ is the angle (mod $2\p$) formed by the direction of (say) the major
 equatorial axis of the satellite  with the  direction of the semi--major axis of the Keplerian ellipse plane;   `dot' represents   derivative with respect to $t$ where
 $t$ (also defined mod $2\p$) is the mean anomaly (i.e.,  the  ellipse area
 between the semi--major axis and
 the orbital radius $\rho_{\ee}$ divided by the total  area times $2\p$)
 and $\ee$ is the eccentricity of the ellipse;

\item[(b)]  the dissipation  parameters ${\eta} = K \Omega_\ee$ and ${\nu}=\nu_\ee$
 are real-analytic functions of the eccentricity $\ee$: $K \geq 0$ is a physical constant depending on the internal (non-rigid) structure of the satellite and\footnote{In \cite{L} (see Eq.ns 2)
$\Omega_\ee$ and $N_\ee$ are denoted, respectively, $\Omega(e)$ and $N(e)$, while, in \cite{peale}, they  are denoted, respectively, by $f_1(e)$ and $f_2(e)$.}
\begin{eqnarray} \label{nu}
\Omega_\ee &:=& \left( 1 + 3\ee^2 + \frac{3}{8}\ee^4 \right)\frac{1}{\left( 1-\ee^2\right)^{9/2} },\nonumber\\
N_\ee &:=& \left( 1 + \frac{15}{2}\ee^2 + \frac{45}{8}\ee^4 + \frac{5}{16}\ee^6 \right)  \frac{1}{\left( 1-\ee^2\right)^{6} },\nonumber\\
%\bar{\nu}_\ee &:=& \frac{N_\ee}{\Omega_\ee}=1+6\ee^2 + \frac{3}{8}\ee^4 + O(\ee^6).
\nu_\ee&:=& \frac{N_\ee}{\Omega_\ee}\ .
\end{eqnarray}

\item[(c)]  the constant ${\varepsilon}$
%, called equatorial ellipticity,
measures the {\sl oblateness} (or ``equatorial ellipticity'') of the satellite and it is defined
as $\e=\frac32\, \frac{B-A}C,$ where  $A\le B$ and $C$  are the principal moments of inertia
of the satellite ($C$ being referred to the polar axis);

\item[(d)]  the function $f$ is the (``dimensionless'') Newtonian potential   given by
\begin{equation}
\label{def_f}
f(x,t):=-\frac{1}{2 \rho_\ee(t)^3} \cos(2x-2{\rm f}_\ee(t)),
\end{equation}
where
$\rho_\ee(t)$ and ${\rm f}_\ee(t)$ are, respectively, the (normalized) orbital radius
\begin{equation}
\label{def_rho} \rho_\ee(t):=1-\ee \cos(u_\ee(t))
\end{equation}
and the polar angle (see\footnote{The analytic expression of the true anomaly in terms of the eccentric anomaly is given by ${\rm f}_\ee(t)= 2 \arctan \left( \sqrt{\frac{1+\ee}{1-\ee}}
\tan\left( \frac{u_\ee(t)}{2}\right) \right)$.} Figure~1);
the eccentric anomaly $u=u_\ee(t)$ is defined implicitly
by the Kepler
equation\footnote{As  well known (see \cite{Wintner41})
$\ee\to u_\ee(t)$ is, for every $t \in \R$, holomorphic for
$|\ee|< r_{\star}$, with
$\dst
%\label{analyticity_of_u}
r_\star := \max_{y \in \R} \frac{y}{\cosh(y)} = \frac{y_\star}{\cosh(y_\star)} =  0.6627434\cdots  \textrm{ and  } y_\star = 1.1996786\cdots .$
}
\begin{equation}
\label{kepler_equation}
t=u-\ee \sin(u).
\end{equation}
Notice that the Newtonian potential $f(x,t)$ is a doubly--periodic function of $x$ and $t$, with periods $2\p$.
\end{enumerate}

\noindent
{\bf Remarks:}
\begin{itemize}
\item[(i)]
The principal moments of an ellipsoid of mass $m$ and with axes
$a$, $b$ and $c$ are given by
$$
A = \frac{1}{5} m (b^2 + c^2),\qquad B = \frac{1}{5} m (a^2 + c^2),
\qquad C
= \frac{1}{5} m (a^2 + b^2).
$$
The oblateness   $\varepsilon$ is then given by
\begin{equation}
\label{oblateness}
\varepsilon = \frac{3}{2} \frac{B-A}{C} = \frac{3}{2} \frac{a^2 - b^2}{a^2 + b^2} .
\end{equation}

\item[(ii)]
There is no universally accepted determination of the internal rigidity constant $K$ for most satellites of the Solar system\footnote{See, however:\\
Iess, L.; et al. (2012). \textit{The tides of Titan}. Science, 337(6093):457-9;
\\
Hussmann, H., Sohl, F., and Spohn, T.. \textit{Subsurface oceans and deep interiors of medium-sized outer planet satellites and large trans-neptunian objects}. Icarus 185, pp. 258-273,  (2006);
%\\
%Lainey, V.; et al.. \textit{Strong tidal dissipation in Io and Jupiter from astrometric observations}. Nature, Volume 459, Issue 7249, pp. 957-959 (2009);
\\
Lainey, V.; et al.. \textit{Strong Tidal Dissipation in Saturn and Constraints on Enceladus' Thermal State from Astrometry}. The Astrophysical Journal, Volume 752, Issue 1, article id. 14, 19 pp. (2012); \\
Castillo-Rogez, J. C., Efroimsky, M. and Lainey, V.;
\textit{The tidal history of Iapetus: Spin dynamics in the light of a refined dissipation model}. Journal of geophysical research, Vol. 116, E09008, doi:10.1029/2010JE003664, 2011.}.
For the Moon and Mercury an accepted value is $\sim 10^{-8}$; see, e.g., \cite{Celletti90}.
However, for our analysis to hold it will be enough that
$\eta\le 0.008$ for the moons and
$\eta\le 0.001$ for Mercury.
\end{itemize}

\nl
The known physical parameter values of the eighteen moons
of the Solar system needed for our analysis are reported in the following table\footnote{$a\ge b$ denote the maximal and minimal observed equatorial radii, which, in our model, are assumed to be the axes of the  ellipse modeling the equatorial section of the satellite; the dimensions of the polar radius are not relevant in our model, however, for all the cases considered in this paper it turns out to be always smaller or equal than the smallest equatorial radius.}:

\Giu

\centerline{
{\small\addtolength{\tabcolsep}{-5pt}
\begin{tabular}{|c|l|c|c|c|c|c|}
\hline
\textbf{Principal}&\textbf{\ \ Satellite\ }&\textbf{Eccentricity}& $a$ & $b$ &  \textbf{Oblateness}  & $\nu$ \\
\textbf{body}& &\textbf{e} & \textbf{(km)}  & \textbf{(km)}  & $\varepsilon = \frac{3}{2} \frac{a^2-b^2}{a^2+b^2} $&   \\
\hline
Earth & Moon$^{\dagger}$ & $0.0549$ & $1740.19$ & $1737.31$ &  $0.00248454179$ & $1.018088056$ \\
\hline
%Mars & Phobos$^{\Join, \ast}$ & $0.0151$ & $13.4$ & $11.2$  &  $-$ & $1.00136808$ \\
%\hline
% & Deimos$^{\Join, \ast}$ & $0.0002$ & $7.5$ & $6.1$  &  $-$ & $1.00000024$ \\
% \hline
 Jupiter & $\textrm{Io}^{\clubsuit}$ & $0.0041$ & $1829.7 $ & $1819.2 $  &  $ 0.00863266715$ & $1.00010086$ \\
  \hline
 & $\textrm{Europa}$ & $0.0094$ & $1561.3 $ & $1560.3 $  &  $ 0.00096104552$ & $1.000530163$ \\
 \hline
 & $\textrm{Ganymede}$ & $0.0011$ & $ 2632.9$ & $ 2629.5$ &  $ 0.0019382783$ & $1.00000726$ \\
 \hline
 & $\textrm{Callisto}$ & $0.0074$ & $ 2411.8$ & $2408.8 $  &  $ 0.00186698679$& $1.000328561$ \\
 \hline
% & Amalthea$^{\ast}$ & $0.0031$ & $ 125$ & $ 73$  &  $-$ & $1.00005766$ \\
% \hline
Saturn & Mimas$^{\spadesuit}$ & $0.0193$ & $208.3 $ & $196.2 $  &  $0.08966019091$ & $1.002234993$ \\
 \hline
 & Enceladus$^{\spadesuit}$ & $0.0047$ & $ 257.2$ & $251.2 $  &  $0.03540026218$ & $1.00013254$ \\
  \hline
& Tethys$^{\spadesuit}$ & $0.0001$ & $ 538.7$ & $527.0 $  &  $0.03293212897$ & $1.00000006$ \\
  \hline
& Dione$^{\spadesuit}$ & $0.0022$ & $ 564.0$ & $560.8 $  &  $0.00853478156$ & $1.00002904$ \\
  \hline
& Rhea$^{\spadesuit}$ & $0.001$ & $ 766.8$ & $761.8 $  &  $0.0098127957$ & $1.000006$ \\
  \hline
& Titan$^{\triangledown}$ & $0.0288$ & $2575.239  $ & $  2574.932$  &  $0.00017882901$ & $1.00497691$ \\
  \hline
& Iapetus$^{\spadesuit}$ & $0.0283$ & $ 748.9 $ & $ 743.1 $  &  $0.011662022156$ & $1.004805592$ \\
  \hline
%& Janus$^{\vartriangle}$ & $0.0073$ & $ 97.4$ & $ 96.9 $  &  $0.00771996946$ & $1.000319741$ \\
%  \hline
%& Epimetheus$^{\vartriangle}$ & $0.0205$ & $ 58.7$ & $ 58.0$  &  $0.01799421119 $ & $1.002521568$ \\
%  \hline
Uranus & Ariel$^{\heartsuit}$ & $0.0012$ & $ 582.0$ & $577.3 $  &  $ 0.012162311957$ & $1.00000864$ \\
  \hline
& Umbriel$^{\heartsuit}$ & $0.0039$ & $ 587.5 $ & $ 581.9 $  &   $0.01436601227$ & $1.00009126$ \\
  \hline
& Titania$^{\heartsuit}$ & $0.0011$ & $ 790.7$ & $787.1 $  &  $0.00684493838$ & $1.00000726$ \\
  \hline
& Oberon$^{\heartsuit}$ & $0.0014$ & $ 764.0 $ & $ 758.8$  &  $0.01024416739$ & $1.00001176$ \\
  \hline
& Miranda$^{\heartsuit}$ & $0.0013$ & $ 241.0 $ & $233.3 $  &  $0.04869051956$ & $1.00001014$ \\
  \hline
Pluto & Charon$^{\intercal}$ & $0.0022$ & $605.0 $ & $602.2 $  & $0.00695821306 $ & $1.00002904$ \\
  \hline
\end{tabular}
}
}
\nopagebreak

\Giu
\centerline{
{\footnotesize  {\bf Table 1.}
Physical data of the moons in 1:1 spin--orbit resonance}}
\nl
{\footnotesize
\centerline{\url{http://ssd.jpl.nasa.gov/?sat_phys_par} and
\url{http://ssd.jpl.nasa.gov/?sat_elem}
}
\nopagebreak
\nl
$^{\dagger}$: Runcorn, S. K.; Hofmann, S.
 Proceedings
from IAU Symposium no. $47$,
Dordrecht, Reidel (1972).
\\
$\clubsuit$: Thomas, P. C.; et al.  Icarus 135 (1998).
\\
$^{\heartsuit}$: Thomas, P. C.;  Icarus 73 (1988).
\\
$^{\spadesuit}$: Dougherty, M.K.; al. (eds.)
 DOI $10.1007/978$-$1$-$4020$-$9217$-$6$\_$24$,
 (2009).
\\
$^{\triangledown}$: Iess, L.; et al.
 Science 327 (2010).
\\
$^{\intercal}$: Sicardy, B., et al.
 Nature 439 (2006) \\
}

\Giu

\nl
The corresponding data of Mercury  are:

\Giu

\centerline{
{\small\addtolength{\tabcolsep}{-5pt}
\begin{tabular}{|c|l|c|c|c|c|c|}
\hline
\textbf{Principal}&\textbf{Satellite}&\textbf{Eccentricity}& $a$ & $b$ & \textbf{Oblateness}  & $\nu$ \\
\textbf{body}& &\textbf{e} & \textbf{(km)}  & \textbf{(km)}  & $\varepsilon = \frac{3}{2} \frac{a^2-b^2}{a^2+b^2} $&   \\
\hline
Sun & Mercury & $0.2056$ & $2440.7$ & $2439.7$ & $0.00061470369$ & $1.255835458$ \\
\hline
\end{tabular}
}
}
\nopagebreak
\Giu
\centerline{{\footnotesize  {\bf Table 2.}
Physical data of Mercury (3:2 spin--orbit resonance)}}
\nopagebreak
\centerline{
{\scriptsize
{\url{http://nssdc.gsfc.nasa.gov/planetary/factsheet/mercuryfact.html} 
%\url{http://solarsystem.nasa.gov/planets/charchart.cfm}
}}}
\centerline{
{\scriptsize
{
%\url{http://nssdc.gsfc.nasa.gov/planetary/factsheet/mercuryfact.html} 
\url{http://solarsystem.nasa.gov/planets/charchart.cfm}
}}}

\Giu

\noindent
{$\bullet$} {\it Existence Theorem for Solar System spin--orbit resonances}

\nl
In this framework, a  $p$:$q$ {\sl spin--orbit  resonance} (with $p$ and $q$   co--prime non--vanishing integers) is, by definition, a solution $t\in \real\to x(t)\in \real$ of \equ{eqn1}
such that
\begin{equation}\label{tommasone}
    x(t+ 2 \pi q) = x(t) + 2 \pi p\ ;
\end{equation}
indeed, for such orbits, after $q$  revolutions of the orbital radius,
 $x$ has made $p$ complete rotations\footnote{Of course,
 in physical space, $x$ and $t$  being angles,
 are defined modulus $2\p$, but to keep track of the topology (windings and rotations) one needs to consider  them in the universal cover $\real$ of $\real/(2\p\integer)$.}.

\nl
Our main result can, now, be stated as follows

\nl

\Giu
{\bf Theorem}
{\rm [moons]}{\sl
The differential equation \equ{eqn1} {\rm (a)$\div$(d)} admits spin--orbit   resonances   \equ{tommasone} with $p=q=1$  provided
{\bf e}, $\nu$
and $\e$ are as in Table 1 and $0\le \eta\le 0.008$.

\nl
{\rm [Mercury]}
The differential equation \equ{eqn1} {\rm (a)$\div$(d)} admits spin--orbit   resonances   \equ{tommasone} with $p=3$ and $q=2$  provided
{\bf e}, $\nu$ and $\e$ are as in Table 2 and $0\le \eta\le 0.001$.
}

\Giu
In \cite{BC09} (compare Theorem~1.2)
 existence of spin--orbit
resonances with $q=1,2,4$ and any $p$ (co--prime with $q$) is proved\footnote{The procedure consisting in reducing the problem
to a fixed point one containing parameters:
the question is then solved by a  Lyapunov--Schmidt
or ``range--bifurcation'' decomposition.
The ``range equation'' is solved by standard contraction mapping methods,
but in order of the fixed point to
correspond to a true solution of the original problem a compatibility
(zero--mean) condition has to be satisfied
(``the bifurcation equation'')
and this is done exploiting a free parameter by means of a topological
argument.}, while in 
 \cite{CC09} quasi--periodic solutions,
corresponding to $p/q$ irrational are studied in the same model.
In  \cite{BC09}
no explicit computations of  constants
(size of admissible $\e$, size of admissible $\eta$, ...)
have been carried out.

\nl
The main point of this paper  is to compute all constants explicitly in order to get nearly 
optimal estimates  and include {\sl all} cases of physical interest.

\section{Proof of the theorem}

\subsubsection*{Step 1. Reformulation of the problem of finding spin--orbit resonances}
Let $x(t)$ be  a $p$:$q$ spin--orbit resonance and let $u(t):= x(qt)-pt-\xi$. Then, by \equ{tommasone} and choosing $\xi$ suitably one sees immediately that  $u$ is $2\p$--periodic and satisfies the differential equation
\beq{u''}
u''(t)+\etar  \big( u'(t)-\nur  \big)
+\epsr  f_x\big(\xi+ pt +u(t), qt\big)=0\ ,\qquad \langle u\rangle = 0,
\eeq
where $\langle \cdot\rangle$ denotes the average over the period\footnote{The parameter $\x$ 
is given by $(1/2\p) \int_0^{2\p} \big( x(qt)- pt\big) dt$
and will be our ``bifurcation parameter''.} and
\begin{equation}\label{etae}
\etar:= q  \eta\,,\qquad
\nur := q \nu -p\,,\qquad
\epsr := q^2  \e\,.
\end{equation}
%and the parameter $\x$
%
%Let $p$ and $q$ be positive co--prime positive integers, then
%\begin{equation}\label{x}
%x(t)=x_{pq}(t)=\x+pt/q+u(t/q)
%\end{equation}
% is a $p\!:\!q$ spin--orbit  resonance for  \eqref{eqn1}
%if and only if  $u$ satisfies
%\beq{u''0}
%u''(t/q)+\eta q \big( {u'}(t/q)+p-\nu q \big)
%+\e q^2 f_x\big(\xi+ pt/q +u(t/q),t\big)=0\,.
%\eeq
%Setting
%\and
%replacing   $t$ with $qt$, \equ{u''0} becomes
Separating the linear part from the non--linear one, we can rewrite \equ{u''} as follows: let
\beq{defLPhi}
\left\{
\begin{array}l
Lu:=
u''+\etar   u'\\ \ \\
\Big[\Phi_\xi(u)\Big](t) :=
\etar\nur
-\epsr  f_x\big(\xi+ pt +u(t), qt\big)
\end{array}\right.
\eeq
then, the differential equation in \equ{u''} is equivalent to
\beq{laura}
Lu=\Phi_\xi(u)\ .
\eeq

\subsubsection*{Step 2. The Green operator ${\cal G}=L^{-1}$}
Let $C^k_{\rm per}$ be the Banach space of $2\pi$--periodic $C^k(\real)$ functions
endowed with the $C^k$--norm\footnote{$\dst \|v\|_{C^k}:=  \sup_{0\le j\le k} \sup_{t\in\real}|D^j v(t)|$.};  let
$\ck{k}$ be  the closed subspace of $C^k_{\rm per}$  formed by functions with vanishing average over $[0,2\pi]$; finally,   denote by
$\bB:=\ck{0}$ the Banach space of $2\p$--periodic continuous functions with zero average (endowed with the sup--norm).

\nl
The linear operator $L$ defined in (\ref{defLPhi}) maps injectively $C^2_{\rm per, 0}$ onto $\bB$; the inverse operator (the ``Green operator'')  ${\cal G}=L^{-1}$
is a bounded linear isomorphism.  
%with inverse $L$. 
Indeed, the following elementary Lemma holds:

\begin{lem}\label{green}
Let $\etar<2/\pi.$ Then\footnote{$\dst \|{\cal G}\|_{L(\bB,\bB)}=\sup_{u: \|u\|_{C^0}=1} \|{\cal G}(u)\|_{C^0}$. }
$$
\|{\cal G}\|_{L(\bB,\bB)}
\leq
\Big(1+\etar\frac{\pi}{2}\big(1-\etar\frac{\pi}{2}\big)^{-1}\Big) \frac{\pi^2}{8}\,.
$$
\end{lem}

\nl
In particular, assuming
\begin{equation}\label{norm green condition}
\etar \leq \frac{\pi}{5}\left( \frac{10}{\pi^2}-1 \right)  , \ {\rm i.e.} \ ,\ \
\eta\le
\left\{
\begin{array}{ll}
\frac{\pi}{5}\left( \frac{10}{\pi^2}-1 \right)=0.0083\cdots \ , & {\rm if}\ (p,q)=(1,1) \\  \ \\
\frac{\pi}{10}\left( \frac{10}{\pi^2}-1 \right)=0.0041\cdots\ ,& {\rm if}\ (p,q)=(3,2)
\end{array}
\right.
\end{equation}
one gets
\beq{bound on G}
\|{\cal G}\|_{L(\bB,\bB)}
\leq \frac54\,.
\eeq
The proof of the above lemma is based on the
following elementary result, whose proof is given
in\footnote{It is easy to see that the estimates in
Lemma~\ref{pesto} are sharp.}   Appendix~A.
\begin{lem}\label{pesto}
\begin{eqnarray}\label{bergoglio}
&&\qquad\qquad v\in
\ck{1}
\qquad\Longrightarrow\qquad
\|v\|_{C^0}\leq \frac{\pi}{2}\|v'\|_{C^0}
\\
\label{bergogliobis}
&&\qquad\qquad
v\in \ck{2}
\qquad\Longrightarrow\qquad
\|v\|_{C^0}\leq \frac{\pi^2}{8}\|v''\|_{C^0}
\end{eqnarray}
\end{lem}

\noindent
{\bf Proof of Lemma \ref{green}}
Given $g\in\bB$ with $\|g\|_{C^0}=1$
we have to prove that if $u\in\ck{2}$ is the unique
solution of
$u''+\etar u'=g$ with $\langle u\rangle=0$,  then
\begin{equation}\label{zucca}
\|u\|_{C^0}
\leq
\Big(1+\etar\frac{\pi}{2}\big(1-\etar\frac{\pi}{2}\big)^{-1}\Big) \frac{\pi^2}{8}\,.
\end{equation}
We note that, setting $v:=u',$ we have that
$v\in\bB$ and
$v'=-\etar v+g$. Then we get
$$
\|v\|_{C^0}
\stackrel{\eqref{bergoglio}}\leq
\frac{\pi}{2}\|-\etar v+g\|_{C^0}
\leq \frac{\pi}{2} (\etar\|v\|_{C^0}+1)\,,
$$
which implies
\begin{equation}\label{pajata}
\|u'\|_{C^0}=\|v\|_{C^0}
\leq
\Big(1-\frac{\pi}{2}\etar\Big)^{-1} \frac{\pi}{2}\,.
\end{equation}
Since $u''=-\etar u' +g$, we have
$$
\|u\|_{C^0}
\stackrel{\eqref{bergogliobis}}\leq
\frac{\pi^2}{8}\|-\etar u'+g\|_{C^0}
\leq
\frac{\pi^2}{8} (1+\etar \|u'\|_{C^0})
$$
and \eqref{zucca} follows by \eqref{pajata}. \eproof

\subsubsection*{Step 3. Lyapunov--Schmidt decomposition}

Solutions of  \equ{laura} are recognized as fixed points of the operator
${\cal G}\circ \Phi_\xi$:
\beq{lauraG}
u={\cal G}\circ \Phi_\xi(u)\ ,
\eeq
where $\xi$ appears as a parameter.

\nl
To solve   equation \equ{lauraG}, we shall perform a {\sl Lyapunov--Schmidt decomposition}.
Let us denote by
$\hat\Phi_\xi: C^0_{\rm per}\to \bB=\ck{0}$ the operator
\beqa{hatPhi}
\hat\Phi_\xi(u)&:=&\frac1\epsr \, \big[ \Phi_\xi(u) -\langle \Phi_\xi(u) \rangle\big]\\
&\phantom{:}=&
- f_x\big(\xi+ pt +u(t), qt\big)
+ \phi_u(\xi) ;
\nonumber\eeqa
where
\beq{defphi}
\phi_u(\xi) :=\frac{1}{2\pi}\int_0^{2\pi} f_x\big(\xi+ pt +u(t;\xi), qt \big)dt\ .
\eeq
Then, equation \equ{lauraG} can be splitted into a
 ``range equation''
\begin{equation}\label{range}
u =\epsr  {\cal G} \circ \hat\Phi_\xi(u)\,,
\end{equation}
(where $u=u(\cdot;\xi)$) and a ``bifurcation  (or kernel) equation''
\begin{equation}\label{kernel}
\phi_u(\xi)=\frac{\etar\nur}\epsr
\quad \Longleftrightarrow\quad
\langle \Phi_\xi\big(u(\cdot;\xi\big) \rangle=0
\,.
\end{equation}

\begin{rem} {\rm (i)}
If $(u,\xi)\in \bB\times[0,2\pi]$
solves
 \equ{range}{\rm \&}\equ{kernel},
then,  $x (t)$  solves \equ{eqn1}.

\noi
{\rm (ii)}
 $\forall \xi\in[0,2\pi]$, $\hat\Phi_\xi\in {C}^1(\bB,\bB)$; indeed, $\forall (u,\xi)\in \bB\times[0,2\pi]$,
\begin{equation}\label{hatPhi1}
\| \hat\Phi_\xi(u)\|_{{C}^0}\leq
 2\sup_{\mathbb{T}^2}|f_x|\,,\quad
\| D_u\hat\Phi_\xi\|_{\mathcal{L}(\bB,\bB)}\leq
 2\sup_{\mathbb{T}^2}|f_{xx}|\,.
\end{equation}
\erem

\nl
The usual way to proceed to solve \equ{range}\&\equ{kernel} is the following:
\begin{enumerate}
\item for any $\xi\in[0,2\pi]$, find $u=u(\cdot;{\xi})$ solving \equ{range};
\item insert $u=u(\cdot,\xi)$ into the kernel equation \equ{kernel} and determine $\xi\in[0,2\pi]$ so that \equ{kernel} holds.
\end{enumerate}

\subsubsection*{Step 4. Solving of the  range equation (contracting map method)}
For $\epsr$ small the range equation is easily solved by standard contraction arguments.

\nl
Let $R:=\frac52\epsr \sup_{\mathbb{T}^2}|f_x|$
and
let \beq{deff} \left\{\begin{array}l
\bB_{R}:=\big\{v\in\bB:\ \|v\|_{C^0}\le R\big\}
\ , \\ \   \\
\varphi:\ v\in \bB_{R}\to
\varphi(v):=\epsr {\cal G} \circ \hat \Phi_\xi(v).
\end{array}\right.
\eeq

\begin{pro}\label{pro:range}
Assume that $\etar$ satisfies \eqref{norm green condition}
and that
\begin{equation}\label{cozza}
\frac52\epsr \sup_{\mathbb{T}^2}|f_{xx}|<1\,.
\end{equation}
Then, for every $\x\in[0,2\p]$, there exists a unique $u:=u(\cdot;\x)\in \bB_{R}$ such that
$\vf(u)=u$.
\end{pro}

\proof
By \eqref{norm green condition} and  \equ{hatPhi1}
 the map $\vf$ in \equ{deff} maps $\bB_{R}$ into itself and
is a contraction with Lipschitz constant
smaller than 1 by \equ{cozza}.
The proof follows by the standard fixed point theorem.
\eproof

\nl
Recalling \equ{def_f}, \equ{def_rho} and
\eqref{etae}, the ``range condition'' \eqref{cozza} writes
\beq{range condition}
 \e
<\left\{
\begin{array}{ll}
\frac{(1-\ee)^3}{5}\ , & {\rm if}\ (p,q)=(1,1), \\  \ \\
 \frac{(1-\ee)^3}{20}\ ,& {\rm if}\ (p,q)=(3,2).
\end{array}
\right.
\eeq

\subsubsection*{Step 5. Solving the bifurcation equation \equ{kernel}}

\nl The function
$\phi_u(\xi)$  in \equ{defphi}
 can be written as
\beqa{phi}
\phi(\xi) &=&
\phi^{(0)}(\xi)+\epsr\tilde\phi^{(1)}_u(\xi;\epsr)
\eeqa
with
\begin{equation}\label{phi0}
\phi^{(0)}(\xi) :=
\frac{1}{2\pi}\int_0^{2\pi}f_x (\xi+pt,qt)dt\,.
\end{equation}
By \eqref{deff}, for $\e$ satisfying
\eqref{range condition},
\beq{esttphi1}
\sup_{
\xi\in[0,2\pi]
} |\tilde\phi^{(1)}_u| \le  \sup_{\mathbb{T}^2}|f_{xx}|
\frac{R}{\epsr}\le
\frac52 \big(\sup_{\mathbb{T}^2}|f_{x}|\big) \big( \sup_{\mathbb{T}^2}|f_{xx}|\big)\,.
\eeq
By\equ{def_f}, \equ{def_rho}, for $\e$ satisfying
\eqref{range condition}, one finds immediately that
\begin{equation}\label{M1}
\sup_{
 \xi\in[0,2\pi]
} |\tilde\phi^{(1)}_u|\le     M_1:=\frac{5}{(1-\ee)^6}\,.
\end{equation}
Let us, now,  have a closer look at the zero order part $\phi^{(0)}$.
The Newtonian potential $f$  has the Fourier expansion
\begin{equation}
\label{fourier_f}
f(x,t) = \sum_{j \in \Z, j \neq 0} \alpha_j \cos(2x - jt),
\end{equation}
where the Fourier coefficients $\a_j=\a_j(\ee)$ coincide with the Fourier coefficients of \begin{equation}
\label{definition_1}
G_\ee(t):= -\frac{e^{2i {\rm f}_\ee(t)}}{2\rho_e(t)^3}=
\sum_{j\in \integer,j\neq 0} \a_j \exp(ijt)\ ,
\end{equation}
(see Appendix~\ref{appB}).
Thus,
$$
f_x (\xi+pt,qt)=-2
\sum_{j\in{\mathbb Z},\ j\neq 0}\a_j
\sin(2\xi+(2p-jq)t)
$$
and one finds:
\beq{phi.so.nd}
\phi^{(0)}(\xi)
=\left\{
\begin{array}{ll}
-2\a_{2}\sin(2\xi)\ , & {\rm if}\ (p,q)=(1, 1), \\  \ \\
-2\a_{3}\sin(2\xi)\ ,& {\rm if}\ (p,q)=(3,2).\end{array}
\right.
\eeq
Define
\beq{apq}
a_{pq}:= \left\{
\begin{array}{ll}
2|\a_{2}|- \hat\e M_1\ , & {\rm if}\ (p,q)=(1,1), \\  \ \\
2|\a_{3}|- \hat\e M_1  \ ,& {\rm if}\ (p,q)=(3,2).\end{array}
\right.
\eeq
Then, from \equ{phi}, \equ{M1},  \equ{phi.so.nd} and \equ{apq}, it follows that
$\phi([0,2\p])$ contains the interval $[-a_{pq},a_{pq}]$,
which is not empty  provided (recall \equ{etae} and \equ{M1})
\beq{non empty condition}
 \e
<\left\{
\begin{array}{ll}
\frac{2(1-\ee)^6}{5}|\a_2(\ee)| \ , & {\rm if}\ (p,q)=(1,1), \\  \ \\
\frac{(1-\ee)^6}{10}|\a_3(\ee)|\ ,& {\rm if}\ (p,q)=(3,2).\
\end{array}
\right.
\eeq
Therefore, we can conclude that  the bifurcation equation \equ{kernel} is solved if one assumes that
$|\frac{\etar\nur}\epsr|\le a_{pq}$, i.e. (recall again \equ{etae}, \equ{M1} and \equ{apq}), if

\beq{bifurcation condition}
 \eta
<\left\{
\begin{array}{ll}
\frac{\e}{| \nu-1|}
\left(
2|\a_2(\ee)|-\frac{5 \e}{(1-\ee)^6}
\right)
 \ , & {\rm if}\ (p,q)=(1,1), \\  \ \\
\frac{2\e}{|2 \nu-3|}
\left(
2|\a_3(\ee)|-\frac{20 \e}{(1-\ee)^6}
\right)
\ ,& {\rm if}\ (p,q)=(3,2).\
\end{array}
\right.
\eeq
We have proven the following:

\begin{proposition}\label{pro:conditions}
Let $(p,q)=(1,1)$ or $(p,q)=(3,2)$ and assume \equ{norm green condition}, \equ{range condition}, \equ{non empty condition} and \equ{bifurcation condition}.
Then, \equ{eqn1} admits $p$:$q$ spin--orbit  resonances $x(t)$ as in \equ{tommasone}.
\end{proposition}

\subsubsection*{Step 6. Lower bounds on $|\a_2(\ee)|$ and $|\a_3(\ee)|$ }
In order to complete the proof of the Theorem, checking the conditions of Proposition~\ref{pro:conditions} for the resonant satellites of the Solar system,
we need to give lower bounds on the absolute values of the Fourier coefficients $\a_2(\ee)$ and $\a_3(\ee)$.
To do this we will simply use Taylor formula to
develop $\alpha_j(\ee)$
in power of $\ee$
up to a suitably large order\footnote{We shall choose $h=4$ for the 1:1 resonances and $h=21$ for the 3:2 case of Mercury.}
\begin{equation}
\label{remainder}
\alpha_j(\ee)=
\sum_{k=0}^h \alpha_j^{(k)} \ee^k
+R_j^{(h)}(\ee)
\end{equation}
and use the analyticity property of $G_\ee$ to get an upper bound on $R_j^{(h)}$ by means of standard Cauchy estimates for holomorphic functions. To use Cauchy estimates, we need an upper bound of $G_\ee$ in a complex eccentricity region. The following simple result will be enough.
\begin{lemma}
\label{paperino}
Fix $0<b<1$. The solution $u_\ee(t)$ of the Kepler equation \equ{kepler_equation}
 is, for every $t\in \real$,  holomorphic with respect to $\ee$ in the complex disk
\begin{equation}\label{checco}
    |\ee|< e_*:= \frac{b}{\cosh b}
\end{equation}
and satisfies
\begin{equation}\label{pluto}
    \sup_{t\in\real} |u_{\ee}(t)-t| \leq b\,.
\end{equation}
Moreover
$\rho_\ee(t)=1-\ee \cos(u_\ee(t))$
satisfies
\begin{equation}\label{bound_rho}
|\rho_\ee(t)|\geq 1-b\,,\qquad \forall\, t\in\mathbb R\,,\
|\ee|<e_*
\end{equation}
and $G_\ee(t)$ (defined in \equ{definition_1}) satisfies
\begin{equation}\label{ziopaperone}
|G_\ee(t)|
\leq
\frac{2}{(1-b)^5}\Big( |1-\ee| (1+\cosh b) +1-b \Big)^2
\,,
\qquad \forall\, t\in\mathbb R\,,\
|\ee|<e_* \,.
\end{equation}
\end{lemma}

\begin{proof}
Using that
\begin{equation}\label{topolino}
\sup_{|\Im z|<b} |\sin z|=\sup_{|\Im z|<b} |\cos z|= \cosh b\,,
\end{equation}
one sees that for $|\ee|< e_*$ the map
$v \mapsto \chi_\ee (v)$ with
$\big[ \chi_\ee (v) \big] (t):=\ee \sin \big( v(t)+t\big)$ is a contraction
in the closed ball of radius $b$ in the space of continuous functions
endowed with the $\sup$-norm. Moreover, since
$\chi_\ee (v)$ is holomorphic in $\ee$, the same holds for
  the fixed point
$v_{\ee}(t)$ of $\chi_\ee$. The estimate in
\eqref{pluto} follows by observing that
$u_{\ee}(t)=v_{\ee}(t)+t$.
Since by \eqref{pluto} we get
\begin{equation}\label{Im_u}
\big|\Im \big(u_{\ee}(t)\big)\big| \leq b
\,,\qquad \forall\, t\in\mathbb R\,,\
|\ee|<e_* \,,
\end{equation}
estimate \eqref{bound_rho} follows by
$$
|\rho_\ee(t)|\geq 1- |\ee| |\cos(u_\ee(t))|
\stackrel{\eqref{topolino}}\geq
1-e_* \cosh b=1-b\,.
$$
Next,  let $\dst w_\ee(t):=\sqrt{\frac{1+\ee}{1-\ee}} \tan
\left( \frac{u_\ee(t)}{2} \right)$ so that
$
{\rm f}_\ee= 2 \arctan w_\ee$. Then\footnote{Use $\dst e^{2iz}=\frac{i- w }{w+i}=-\frac{(w-i)^2}{w^2+1}$ and $\tan^2 (\alpha/2)=
(1-\cos \alpha)/(1+\cos \alpha)$.},
\begin{equation*}
\label{numeratore}
|e^{2i {\rm f}_\ee(t)}|
= \frac{|w-i|^4 }{|w^2+1|^2}
\leq \left( \frac{4}{|w^2+1|}+2 \right)^2
=
4\left( \frac{|1-\ee||1+\cos u_\ee|}{|1-\ee \cos u_\ee|}+1 \right)^2\ .
\end{equation*}
Then \eqref{ziopaperone}
follows by \eqref{bound_rho},
\eqref{topolino} and \eqref{Im_u}. \eproof
\end{proof}

\begin{lemma}
\label{estimate remainder}
Let $R^{(h)}_j(\ee)$ be as in \equ{remainder}, $0<b<1$ and $0<\ee<b/\cosh b$. Then,
$$    |R_j^{(h)}(\ee)|  \leq
R^{(h)}(\ee;b)
$$
with
$$
R^{(h)}(\ee;b):= \frac{2}{(1-b)^5}\Big( (1+\frac{b}{\cosh b}-\ee) (1+\cosh b) +1-b \Big)^2
\frac{\ee^{h+1}}{(\frac{b}{\cosh b}-\ee)^{h+1}}\ .
$$
\end{lemma}
\begin{proof}
For $\ee,\rho>0$ we set
$$
[0,\ee]_{\rho}:=\{\  z\in\mathbb C\,,\ \ {\rm s.t.}\ \
z=z_1+z_2\,,\ \ z_1\in[0,\ee]\,,\ \
|z_2|<\rho\
\}\,.
$$
Lemma~\ref{paperino} and   standard (complex) Cauchy estimates
imply, for $0\leq s\leq 1$,
$$
|{D^{h+1}}\alpha_j(s\ee)|
\leq \frac{(h+1)!}{(e_*-\ee)^{h+1}}
\sup_{[0,\ee]_{e_*-\ee}}|\alpha_j|
$$
and, therefore,
$$
|R_j^{(h)}(\ee)|
\leq
\frac{\ee^{h+1}}{(e_*-\ee)^{h+1}}
\sup_{[0,\ee]_{e_*-\ee}}|\alpha_j|\ .
$$
By \eqref{ziopaperone}
we obtain
$$
\sup_{[0,\ee]_{e_*-\ee}}|\alpha_j|
\leq
\frac{2}{(1-b)^5}\big( (1+e_*-\ee) (1+\cosh b) +1-b \big)^2
$$
from which,  recalling \eqref{checco}, the lemma follows.
\eproof
\end{proof}

\nl
Now, in order to check the conditions of Proposition~\ref{pro:conditions} we will expand  $\a_2$ in power of $\ee$ up
to order $h=4$ and $\a_3$ up to order $h=21$ . Using the representation formula \equ{gricia} for the $\a_j$ given in Appendix~\ref{appB} we find:
\beqano
\alpha_2(\ee)&=&
- \frac{1}{2} +
 \frac{5}{4} \ee^2 - \frac{13}{32} \ee^4 +
 R^{(4)}_2(\ee)\ ,\\
\alpha_3(\ee)
&=&
-\frac{7}{4} \ee + \frac{123}{32} \ee^3 - \frac{489}{256} \ee^5 + \frac{1763}{4096} \ee^7 - \frac{13527}{327680} \ee^9+ \frac{180369}{13107200} \ee^{11} \\
& & + \frac{5986093}{734003200} \ee^{13} + \frac{24606987}{3355443200} \ee^{15} + \frac{33790034193}{5261334937600} \ee^{17} \\
& &+\frac{1193558821627}{210453397504000} \ee^{19} +
\frac{467145991400853}{92599494901760000}\ee^{21} +
R^{(21)}_3(\ee) \ .
\eeqano
In view of Lemma~\ref{estimate remainder},
we choose, respectively,  $b=0.462678$ and\footnote{The values for $b$ are rather arbitrary (as long as $0<b<1$); our choice is made for optimizing the estimates.} $b=0.768368$ to get lower bounds:
\beqa{cc1}
|\a_2(\ee)|&\ge& \Big| \frac{1}{2} -  \frac{5}{4} \ee^2 +
\frac{13}{32} \ee^4\Big| - |R^{(4)}(\ee;0.462678)|\\
|\alpha_3(\ee)| &\ge & \left|\sum_{k=1}^{21}
\alpha_3^{(k)}
 \ee^k\right| - | R^{(21)}(\ee;0.768368)|\ .
\label{cc2}
\eeqa

\subsubsection*{Step 7. Check of the conditions and conclusion of the proof}
We are now ready to check all conditions of Proposition~\ref{pro:conditions} with
the parameters of the satellite in spin--orbit resonance given in Table 1 and 2.

\nl
In the following Table we report:

\begin{itemize}
\item[] in column 2:
 the lower bounds on $|\a_q(\ee)|$ as obtained in Step~6 using \equ{cc1} and \equ{cc2} (with the eccentricities listed in Table 1 and 2);

\item[] in column 3: the difference between the right hand side and the left hand side of the inequality\footnote{Thus, the inequality is satisfied if the numerical value in the column is positive; the same applies to the 5th column.} \equ{range condition};

\item[] in column 4: the difference between the right hand side and the left hand side of the inequality \equ{non empty condition};

\item[] in column 5:  the right hand side  of the inequality \equ{bifurcation condition}, which is an upper bound for the admissible values of the dissipative parameter $\eta$.

\end{itemize}

\vglue1.truecm

\centerline{
\begin{tabular}{|c|c|c|c|c|}
\hline
\textbf{Satellite}&\textbf{lower bound}& \textbf{r.h.s. -- l.h.s.}  & \textbf{r.h.s. -- l.h.s.} & \textbf{r.h.s. } \\
& \textbf{on $|\alpha_q|$} & \textbf{of Eq. \equ{range condition}} & \textbf{of Eq. \equ{non empty condition}}  & \textbf{of Eq. \equ{bifurcation condition}}  \\
\hline
Moon & $0.45475265$ & $0.1663508$& $0.127144$ & $0.1225335$ \\
 \hline
%Phobos & \\
% \hline
% Deimos & \\
%  \hline
  Io & $0.49997893$ & $0.1889174$& $0.186489$ & $81.800325$ \\
   \hline
   Europa & $0.49988598$ & $0.1934518$& $0.187978$ & $1.8031043$ \\
    \hline
    Ganymede & $0.49999849$ & $0.1974024$& $0.196745$ & $264.3751$ \\
     \hline
     Callisto & $0.49993049$ & $0.1937258$& $0.189389$ & $5.6260606$\\
      \hline
%      Amalthea & \\
%      \hline
       Mimas & $0.49938883$ & $0.0989819$& $0.088051$ & $19.852395$ \\
        \hline
        Enceladus & $0.49997228$ & $0.161793$& $0.159015$ & $218.44519$ \\
         \hline
         Tethys & $0.49999999$ & $0.1670079$& $0.166948$ & $458437.46$ \\
          \hline
          Dione & $0.49999395$ & $0.1901481$& $0.188837$ & $281.18521$ \\
           \hline
           Rhea & $0.49999875$ & $0.1895878$& $0.18899$ & $1554.7362$ \\
            \hline
            Titan & $0.49776167$ & $0.1830341$& $0.166905$ & $0.0357326$ \\
             \hline
             Iapetus & $0.49790449$ & $0.171834$& $0.155986$ & $2.2484865$ \\
              \hline
   %           Janus & $0.4999324$ & $0.1879319$& $0.183652$ & $23.167321$ \\
   %            \hline
   %            Epimetheus & $0.49927518$ & $0.1699562$& $0.158377$ & $6.3987689$  \\
   %             \hline
                Ariel & $0.4999982$ & $0.1871186$& $0.186401$ & $1321.448$ \\
                 \hline
                 Umbriel & $0.499998095$ & $0.1833031$& $0.180992$ & $145.83674$ \\
                  \hline
                  Titania & $0.49999849$ & $0.1924958$& $0.191838$ & $910.34423$  \\
                   \hline
                   Oberon & $0.49999755$ & $0.188917$& $0.188081$ & $826.10305$ \\
                    \hline
                    Miranda & $0.49999789$ & $0.1505305$& $0.149754$ & $3623.6286$ \\
                     \hline
                     Charon & $0.49999395$ & $0.1917247$& $0.190414$ & $231.15781$ \\
                      \hline
                      Mercury & $0.27$ & $0.0244515$& $0.006171$ & $0.0012363$ \\
                       \hline
\end{tabular}
}

\nopagebreak

\Giu

\centerline{\footnotesize  {\bf Table 3.} Check of the hypotheses of Proposition~\ref{pro:conditions} for the satellites in spin--orbit resonance
}

\Giu

\noi
The positive value reported in the third and fourth column means that the range condition \equ{range condition}
and the topological condition \equ{non empty condition} are satisfied for all the moons in 1:1 resonance and for Mercury;
the bifurcation condition \equ{bifurcation condition} yields an upper bound on the admissible value for $\h$ (fifth column). Thus, $\h$ has to be smaller than the minimum between the value in the fifth column of Table~3 and the value in the right hand side of
Eq. \equ{norm green condition} (needed to give a bound on the Green operator): such minimum values is seen to be $0.008\cdots$  for the moons in 1:1 resonance and $0.001\cdots$ for Mercury.

\nl
The proof of
the Theorem is complete.
\eproof

\bigskip

\appendix

\section{Proof of Lemma \ref{pesto}}
\label{Appendix_B}

\proof
We first prove \eqref{bergoglio}.
Up to a rescaling we can prove \eqref{bergoglio}
assuming  $\|v'\|_{C^0}=1.$
Assume by contradiction that
$$\|v\|_{C^0}=:c>\pi/2\,.$$
Note that it is obvious that $c\leq \pi$ since $v$
has zero average and, therefore, it must vanish
at some point.
Up to a translation we can assume that $|v|$
attains maximum in $-c$.
In case, multiplying by $-1$, we can also assume
that $-c$ is a minimum namely
$$
\|v\|_{C^0}=c=-v(-c)\,.
$$
Since $\|v'\|_{C^0}=1$ we get
$$
v(t)\leq -c +|t+c| \qquad \forall t\in[-2c,0]
$$
and, therefore,
\begin{equation}\label{ventaglietto2}
v(0)\leq 0\,,\ \ v(-2c)\leq 0\,,\ \ \ \ \
\int_{-2c}^0 v\leq -c^2\,.
\end{equation}
Since $\|v'\|_{C^0}=1$
we also get
$$
v(t)\leq \pi-c -|t-\pi+c|\qquad
\forall\, t\in [0,2\pi -2c]\,.
$$
Then
$$
\int_0^{2\pi-2c}v\leq (\pi-c)^2\,.
$$
Combining with the last inequality in \eqref{ventaglietto2}
we get
$$
\int_{-2c}^{2\pi-2c}v\leq (\pi-c)^2-c^2=\pi (\pi-2c)<0\,,
$$
which contradicts the fact that $v$ has zero average,
proving \eqref{bergoglio}

\nl
We now prove \eqref{bergogliobis}.
Up to a rescaling we can prove \eqref{bergogliobis}
assuming  $\|v''\|_{C^0}=1.$
Assume by contradiction that
\begin{equation}\label{cesello}
\|v\|_{C^0}=:c>\pi^2/8\,.
\end{equation}
Up to a translation we can assume that $|v|$
attains maximum at $0$.
In case, multiplying by $-1$, we can also assume
that $-c$ is a minimum namely
$$
\|v\|_{C^0}=c=-v(0)\,.
$$
Since $\|v''\|_{C^0}=1$ we get
$$
v(t)\leq -c +t^2/2 \qquad \forall t\in\mathbb R\,.
$$
Since $v$ has zero average must exist
\begin{equation}\label{battisti}
t_1\leq -\sqrt{2c}\,,\ \ t_2\geq \sqrt{2c}\,,\ \
t_2-t_1<2\pi\,,
\ \ \
{\rm s.t.}\ \ \
v(t_1)=v(t_2)=0\,,\ \ v(t)<0 \ \ \forall\, t\in(t_1,t_2)\,.
\end{equation}
Moreover
$$
\int_{t_1}^{t_2}v
\leq
\int_{-\sqrt{2c}}^{\sqrt{2c}} -c+\frac{1}{2}t^2=
-\frac23 (2c)^{3/2}\,.
$$
Since $v$ has zero average and is $2\pi$-periodic
\begin{equation}\label{ventaglietto}
\int_{t_2}^{2\pi+t_1} v=
-\int_{t_1}^{t_2}v
\geq
\frac23 (2c)^{3/2}
\,.
\end{equation}
Set
$$
a:=\pi+ (t_1-t_2)/2
$$
and note that
\begin{equation}\label{deandre}
0<a\leq \pi-\sqrt{2c}<\pi/2\,, \qquad
a^2<2c
\end{equation}
by \eqref{battisti} and \eqref{cesello}.
Set
$$
u(t):=v\big(t+\pi+(t_1+t_2)/2\big)\,.
$$
Note that
$u\in\mathbb B\cap C^2$ and, by \eqref{battisti},
$$
\|u\|_{C^0}=c\,,\ \ \|u''\|_{C^0}=1\,,\ \
u(-a)=u(a)=0\,,\ \
\int_{-a}^a u
=
\int_{t_2}^{2\pi+t_1} v
\stackrel{\eqref{ventaglietto}}\geq
\frac23 (2c)^{3/2}\,.
$$
Consider now the even function
$$
w(t):=\frac12(u(t)+u(-t))\,.
$$
Note that $w\in\mathbb B\cap C^2$
and
\begin{equation}\label{battiato}
\|w\|_{C^0}\leq c\,,\ \ \|w''\|_{C^0}\leq 1\,,\ \
w(-a)=0\,,\ \
\int_{-a}^0 w
=\frac12
\int_{-a}^{a} u
\geq
\frac13 (2c)^{3/2}\,.
\end{equation}
Set
$$
z(t):=c-\frac{c}{a^2}t^2\,.
$$
We claim that
\begin{equation}\label{degregori}
z(t)\geq w(t)\qquad \forall\, -a\leq t\leq 0\,.
\end{equation}
Then
$$
\int_{-a}^0 w\leq \int_{-a}^0 z= \frac23 ca
\stackrel{\eqref{deandre}}<
\frac13 (2c)^{3/2}
\stackrel{\eqref{battiato}}
\leq
\int_{-a}^0 w\,,
$$
which is a contradiction.

\nl
Let us prove the claim in \eqref{degregori}.
Note that $z(-a)=w(-a)=0.$
Assume by contradiction that there exists
$ \bar t\in[-a,0)$
such that
$$
z(\bar t)=w(\bar t)\,,\qquad
z(t)\geq w(t)\ \ \forall\, t\in[-a,\bar t]\,,\qquad
z'(\bar t)\leq w'(\bar t)\,.
$$
Then, since $\|w''\|_{C^0}\leq 1$
\begin{eqnarray*}
w(t)
&\geq&
w(\bar t) +w'(\bar t)(t-\bar t) -\frac12 (t-\bar t)^2
\\
&\stackrel{\eqref{deandre}}>&
z(\bar t) +z'(\bar t)(t-\bar t) -\frac{c}{a^2} (t-\bar t)^2
=z(t)
\,,\qquad \forall\, t\in (\bar t,0]\,.
\end{eqnarray*}
Then
$$
w(0)>z(0)=c\,,
$$
which contradicts the first inequality in \eqref{battiato}.
This completes the proof of \eqref{bergogliobis}. \eproof

\section{Fourier coefficients of the Newtonian potential}
\label{appB}
Properties of the Fourier coefficients $\a_j$ of the Newtonian potential $f$, including Eq. \equ{definition_1},  have been discussed, e.g., in  Appendix~A of\footnote{A factor $-1/2$ is missing in the definition of $G(t)$ given in \cite{BC09},  (iii) p. 4366 and, consequently, it has to be included at p. 4367 in line 6 (from above, counting also lines with formulas) in front of ``Re''; in  line 12, 17 and 18 the factor $1/(2\p)$ has to be replaced by  $-1/(4\p)$.} \cite{BC09}.

\nl
Here we provide a simple formula for the Fourier coefficients $\a_j$ of the Newtonian potential $f$ in \equ{def_f}
(compare (d)  of \S 1, and \equ{fourier_f}--\equ{definition_1}); namely we prove that
\begin{equation}\label{gricia}
\alpha_j=
-\frac{1}{4 \pi} \int_0^{2 \pi}
\frac{1}{\rho^2(w^2+1)^2}
\Big[
(w^4-6w^2+1) c_j (u)
-4w(w^2-1)s_j (u)
\Big]
du\,,
\end{equation}
where $w=w(u;\ee)
:=\sqrt{\frac{1+\ee}{1-\ee}} \tan \frac{u}{2}$,
$\rho=1-\ee\cos u$
and
$$
c_j(u):=\cos(ju-j\ee\sin u)\,,\qquad
s_j(u):=\sin(ju-j\ee\sin u)\,.
$$
\begin{proof}
If $z=\arctan w $, then
\begin{equation}\label{polpettone}
 e^{2iz}=\frac{i- w }{w+i}=-\frac{(w-i)^2}{w^2+1}\,,
\end{equation}
so that if $w_\ee(t):=w(u_{\ee}(t),\ee)$
one has
$
f_\ee= 2 \arctan w_\ee
$
and
\begin{eqnarray}\label{pizza}
G_\ee
&=&
-\frac{1}{2\rho_\ee^3}
\frac{(w_\ee-i)^2}{(w_\ee+i)^2}
=
-\frac{1}{2\rho_\ee^3}
\frac{(w_\ee-i)^4}{(w_\ee^2+1)^2}
\\
&=&
-\frac{1}{2\rho_\ee^3}
\frac{1}{(w_\ee^2+1)^2}
\Big(
w_\ee^4-6w_\ee^2+1
-4iw_\ee(w_\ee^2-1)
\Big)
\,.
\nonumber
\end{eqnarray}
By parity properties, it is easy to see that the
$G_j$'s are real, namely $G_j=\bar G_j$, so that
\begin{eqnarray*}
\alpha_j &=& G_j =
\frac{1}{2 \pi} \int_0^{2 \pi} G(t) e^{-ijt} \, dt
= -\frac{1}{4 \pi} \int_0^{2 \pi} \frac{e^{i2f_\ee(t) - ijt}}{\rho_\ee(t)^3} \, dt \\
&=&
-\frac{1}{4 \pi} \int_0^{2 \pi}
\frac{1}{\rho_\ee^3(w_\ee^2+1)^2}
\Big[
(w_\ee^4-6w_\ee^2+1)\cos (jt)
-4w_\ee(w_\ee^2-1)\sin (jt)
\Big]
dt\,.
\end{eqnarray*}
Making the change of variable given by the Kepler equation \eqref{kepler_equation}, i.e.  integrating from $t$ to $u=u_\ee$ and setting $u_\ee(t)' = \frac{1}{\rho_\ee(t)}$ one gets \equ{gricia}.
\eproof
\end{proof}

\section{Small bodies}
\label{minor bodies}
In the Solar system besides the eighteen moons listed in Table~1 and Mercury there are other five minor bodies with
mean radius smaller than $100$ km observed in 1:1 spin--orbit resonance around their planets:   Phobos and Deimos (Mars), Amalthea (Jupiter), Janus and Epimetheus (Saturn).

\Giu

\Giu
\centerline{
{\small
\begin{tabular}{|c|l|c|c|c|c|c|c|}
\hline
\textbf{Principal}&\textbf{Satellite}&\textbf{Eccentricity}& $a$ & $b$ & \textbf{Oblateness}  & $\nu$ \\
\textbf{body}& &\textbf{e} & \textbf{(km)}  & \textbf{(km)}  & $\varepsilon = \frac{3}{2} \frac{a^2-b^2}{a^2+b^2} $&   \\
\hline
Mars & Phobos$^{\ast,\Join}$ & $0.0151$ & $13.4$ & $11.2$  &  $0.26616393443$ & $1.00136808$ \\
\hline
 & Deimos$^{\ast,\Join}$ & $0.0002$ & $7.5$ & $6.1$  &  $0.30558527712$ & $1.00000024$ \\
 \hline
Jupiter  & Amalthea$^{\ast}$ & $0.0031$ & $ 125$ & $ 73$  &  $0.73704304667$ & $1.00005766$ \\
 \hline
Saturn
& Janus$^{\vartriangle}$ & $0.0073$ & $ 97.4$ & $ 96.9 $  &  $0.00771996946$ & $1.000319741$ \\
  \hline
& Epimetheus$^{\vartriangle}$ & $0.0205$ & $ 58.7$ & $ 58.0$  & $0.01799421119 $ & $1.002521568$ \\
  \hline
\end{tabular}
}
}
\nopagebreak

\giu
\centerline
{\footnotesize  {\bf Table 4.}
Physical data of minor bodies in 1:1 spin--orbit resonance}

\noi
{\footnotesize
$^{\ast}$: Thomas, P. C., et al. Icarus 135 (1998).
\\
$^{\Join}$: Thomas, P. C. Icarus 77 (1989)
and
\url{http://solarsystem.nasa.gov/planets/profile.cfm?Object=Mars&Display=Sats}
\\
$^{\vartriangle}$: Porco, C. C.; et al. Science 318 (2007).
}

\Giu

\noi
Besides being small, such bodies have also a quite irregular shape and only  Janus and Epimetheus have a good equatorial symmetry\footnote{For pictures, see:
\url{http://photojournal.jpl.nasa.gov/catalog/PIA10369}
(Phobos),
\url{http://photojournal.jpl.nasa.gov/catalog/PIA11826}
(Deimos),
\url{http://photojournal.jpl.nasa.gov/catalog/PIA02532}
(Amalthea),
\url{http://photojournal.jpl.nasa.gov/catalog/PIA12714}
(Janus),
\url{http://photojournal.jpl.nasa.gov/catalog/PIA12700}
(Epimetheus).
}. Indeed, for these two small moons (and only for them among the minor bodies), our theorem holds as shown by the data
reported in the following table\footnote{Positive values in the 3rd and 4th column and values less than 0.008 in the 5th column  implies that the  assumptions of Proposition~\ref{pro:conditions} hold.}:

\Giu

\centerline{
{\small
\begin{tabular}{|c|c|c|c|c|}
\hline
\textbf{Satellite}&\textbf{lower bound}& \textbf{r.h.s. -- l.h.s.}  & \textbf{r.h.s. -- l.h.s.} & \textbf{r.h.s.} \\
& \textbf{on $|\alpha_q|$} & \textbf{of Eq. \equ{range condition}} & \textbf{of Eq. \equ{non empty condition}}  & \textbf{of Eq. \equ{bifurcation condition}}  \\
\hline
%Phobos & $0.4996743$ & $-0.075088$& $-0.08373$ & $-89.23777$ \\
% \hline
% Deimos & $0.49999995$ & $-0.105705$& $-0.10583$ & $-674530.2$\\
%  \hline
%   Amalthea & $0.49998797$ & $-0.538897$& $-0.54074$ & $-35210.01$ \\
%    \hline
              Janus & $0.4999324$ & $0.1879319$& $0.183652$ & $23.167321$ \\
               \hline
               Epimetheus & $0.49927518$ & $0.1699562$& $0.158377$ & $6.3987689$  \\
                \hline
\end{tabular}
}
}
\nopagebreak

\Giu
\centerline{\footnotesize  {\bf Table 5.}
Check of the hypotheses of
Proposition~\ref{pro:conditions} for the small
satellites in spin--orbit resonance
}

\Giu


\begin{thebibliography}{1}

\bibitem{BH1} Bambusi, D.; Haus, E.: {\it Asymptotic stability of synchronous orbits for a gravitating viscoelastic sphere}. Celestial Mech. Dynam. Astronom. \textbf{114}, no. 3, 255--277 (2012).

%\bibitem{BH2} Bambusi, D.; Haus, E.:

\bibitem{BC09} Biasco, L.; Chierchia, L.: {\it Low-order resonances in weakly dissipative spin-orbit models}.
Journal of Differential Equations \textbf{246}, 4345-4370 (2009).


%\bibitem{Cayley1860} Cayley, A.: {\it Tables of the developments of functions in the theory of elliptic motion}.
%Mem. Roy. Astron. Soc. \textbf{29}, 191-304 (1860).

%\bibitem{Celletti89} Celletti, A.: {\it Analysis of Resonances in the Spin-Orbit Problem in Celestial Mechanics}. Ph.D. Thesis, ETH-Z\"urich (1989).

\bibitem{Celletti90} Celletti, A.: {\it Analysis of resonances in the spin-orbit problem in Celestial Mechanics: The synchronous resonance (Part I)}.
J. Appl. Math. Phys. (ZAMP) \textbf{41}, 174-204 (1990).

\bibitem{Celletti10_libro} Celletti, A.: {\it Stability and chaos in celestial mechanics}. In Stability and Chaos in Celestial Mechanics. Springer-Praxis, Providence, RI, 2010.

\bibitem{CC09} Celletti, A; Chierchia, L.: {\it
Quasi-periodic attractors in celestial mechanics}
Arch. Rational Mech. Anal., \textbf{191}, Issue 2, 311-345 (2009)

\bibitem{L}  Correia, A. C. M.;  Laskar, J.: \textit{Mercury's capture into the 3/2 spin--orbit resonance as a result of its chaotic dynamics},
Nature {\bf 429} (24 June 2004), 848--850 .

\bibitem{Danby62} Danby, J. M. A.: {\it Fundamentals of Celestial Mechanics}. Macmillan, New York 1962.

\bibitem{Goldreich-Peale67} Goldreich, P.; Peale, S.: {\it Spin-orbit coupling in the solar system}.
Astronomical Journal \textbf{71}, 425 (1967).

%\bibitem{Kaula66} Kaula, W. M.: {\it Theory of Satellite Geodesy}. Dover Publications Inc., New York 1966.

\bibitem{MD}
MacDonald, G.J.F.: \textit{Tidal friction}, Rev. Geophys. {\bf 2} (1964), 467--541.

\bibitem{peale} Peale, S.J.: \textit{The free precession and libration of Mercury},
Icarus {\bf 178} (2005), 4--18.

%\bibitem{Tisserand1889} Tisserand, F.: {\it Trait\' e de M\' ecanique Celeste, Vol 1: Perturbations des Planetes d'apr\' es la Methode de la Variation des Constantes Arbitraires}. Paris: Gauthier-Villars et Fils, 1889, republished, 1960.

%\bibitem{Waldvogel84} Waldvogel, J.: {\it Der Tayloralgorithmus}. Angew. Math. Phys., \textbf{35}(6), 780-789 (1984).

\bibitem{Wintner41} Wintner, A.: {\it The Analytic Foundations of Celestial Mechanics}.
Princeton Univ. Press, Princeton, NJ, 1941.


%\bibitem{Wisdom86} Wisdom, J.: {\it Chaotic dynamics in the solar system}.
%Urey lecture (1986).

\bibitem{Wisdom87_r} Wisdom, J.: {\it Rotational dynamics of irregularly shaped natural satellites}.
Astron. J. \textbf{94}, no. 5, 1350-1360 (1987).
\end{thebibliography}
\end{document}